\documentclass[12pt,reqno]{amsart}

\usepackage{amsmath, amsthm, amssymb,amsfonts, commath}

\usepackage[colorlinks,linkcolor=blue,
anchorcolor=blue,citecolor=blue,backref=page]{hyperref}

\DeclareMathOperator{\lcm}{lcm}
\DeclareMathOperator{\ord}{ord}

\DeclareMathOperator{\rad}{rad}
\DeclareMathOperator{\Orb}{Orb}
\DeclareMathOperator{\dist}{dist}
\newcommand{\N}{\mathbb{N}}
\newcommand{\Z}{\mathbb{Z}}
\newcommand{\Q}{\mathbb{Q}}

\renewcommand{\subset}{\subseteq}
\renewcommand{\supset}{\supseteq}
\newcommand{\lan}{\langle}
\newcommand{\ran}{\rangle}

\usepackage[nameinlink,capitalize]{cleveref}

\numberwithin{equation}{section}

\newtheorem{thm}{Theorem} [section]
\newtheorem{lem}[thm]{Lemma}

\newtheorem{cor}[thm]{Corollary}
\newtheorem{rem}[thm]{Remark}

\theoremstyle{definition}

\begin{document}
	
\title{Rational numbers in $\times b$-invariant sets}

\author{Bing Li} 
\email{scbingli@scut.edu.cn}
\author{Ruofan Li} 
\email{liruofan@scut.edu.cn}
\author{Yufeng Wu}
\email{yufengwu@scut.edu.cn}

\address{School of Mathematics, South China University of Technology, Guangzhou, China, 510641}

\keywords{Cantor sets, rational numbers, $T_b$-invariant sets}

 \thanks{2010 {\it Mathematics Subject Classification}: 11A63, 37E05}

\begin{abstract}
Let $b \geq 2$ be an integer and $S$ be a finite non-empty set of primes not containing divisors of $b$. For any non-dense set $A \subset [0,1)$ such that $A \cap \Q$ is invariant under $\times b$ operation, we prove the finiteness of rational numbers in $A$ whose denominators can only be divided by primes in $S$. A quantitative result on the largest prime divisors of the denominators of rational numbers in $A$ is also obtained. 
\end{abstract}

\maketitle

\section{Introduction}
Let $C$ be the classical middle-third Cantor set, which consists of real numbers in $[0,1]$ whose ternary expansions do not contain digit $1$. In 1984, Mahler \cite{Mahler84} asked how close can irrational elements in $C$ be approximated by rational numbers in $C$. A related question is what are the rational numbers in $C$. For any $n \geq 1$, we know that there are exactly $2^{n+1}$ rational numbers of the form $\frac{a}{3^n}$ in $C$ with $a \in \Z$. One may ask what happens if the denominator is a $d$-power for some $d \geq 2$ instead of a $3$-power. For $d=2$, Wall \cite{Wall83} proved that $\frac{1}{4}$ and $\frac{3}{4}$ are the only dyadic rationals in $C$. More generally, let $S$ be a finite set of primes,  then the set of \textit{$S$-integers} $\Z_{S}$ is defined to be the set of rational numbers whose denominators can only be divided by primes in $S$. Equivalently, $$\Z_{S} = \{\alpha \in \Q \colon v_{p}(\alpha) < 0 \text{ implies } p \in S\},$$ where $v_{p}(\alpha)$ is the unique integer such that $\alpha = p^{v_{p}(\alpha)} \frac{m}{n}$ for some $m,n \in \Z$ coprime with $p$. One may wonder what does the set $\Z_{S} \cap C$ look like. When $S = \{2,5\}$, it was proved by Wall \cite{Wall90} that $\Z_{\{2,5\}} \cap C$ consists of exactly $14$ elements. Later, Nagy \cite{Nagy01} showed that, if $S = \{p\}$ for some prime $p > 3$, then $C$  contains only finitely many $S$-integers. Recently, based on a heuristic argument as well as numerical evidence, Rahm, Solomon, Trauthwein and Weiss \cite{RSTW20} formulated an asymptotic for the number $N^*(T)$ of reduced rational numbers in $C$ with denominators bounded by $T$. 

The distribution of rational numbers could also be studied  for generalized Cantor sets. Let $b \geq 2$ be an integer and $\mathcal{D}$ be a non-empty subset of $\{0,1,\ldots,b-1\}$, the generalized Cantor set $C(b,\mathcal{D})$ is the set of real numbers in $[0,1]$ whose base $b$ expansions only consist of digits in $\mathcal{D}$. Extending Nagy's work, Bloshchitsyn \cite{Bloshchitsyn15} proved that for any integer $b \geq 3$, $\mathcal{D}$ with cardinality  $b-1$ and prime $p > b^2$, the set $\Z_{\{p\}} \cap C(b,\mathcal{D})$ is finite. A very recent work of Schleischitz \cite[Corollary 4.4]{Schleischitz21} showed that $C(b,\mathcal{D})$  contains only finitely many $S$-integers if no element of $S$ divides $b$ and $\mathcal{D}$ has cardinality at most $b-1$. Shparlinski \cite{Shparlinski21} proved a quantitative strengthen of Schleischitz's result. To state Shparlinski's result, we first introduce some notations. For an integer $d \geq 2$, denote the {\em largest prime divisor} of $d$ by $P(d)$ and define the {\em radical} of $d$ by $$\rad(d) = \prod_{p \mid d: \ p\text{ prime}} p.$$

\begin{thm}[\cite{Shparlinski21}] \label{thm:Shp}
Let  $b \geq 2$ be an integer and $\mathcal{D} \subset \{0,1,\ldots,b-1\}$ be a non-empty set of cardinality at most $b-1$. Then there exists a constant $c_{b}>0$, depending only on $b$, such that for any rational number $\frac{a}{d}$ in $C(b,\mathcal{D})$ with $\gcd(ab,d) = 1$, we have 
\begin{align*}
\rad(d) \geq c_{b} \log d \quad \text{ and } \quad P(d) \geq c_b \sqrt{\log d \log \log d}.
\end{align*}
\end{thm}

In this article, we investigate $S$-integers in $T_{b}$-invariant sets. For any integer $b \geq 2$, the transformation $T_{b} \colon [0,1) \to [0,1)$ is defined by $$T_b(x) = bx \pmod{1}.$$ 
We say that a set $A \subset [0,1)$ is \textit{$T_{b}$-invariant} if $T_{b}(A) \subset A$. Clearly all generalized Cantor sets $C(b,\mathcal{D})$ are $T_{b}$-invariant. Our finiteness result is as follows.

\begin{thm} \label{thm:finite}
Let $b \geq 2$ be an integer, $S$ be a non-empty finite set of primes not containing any prime divisor of $b$, and $A$ be a subset of $[0,1)$. If $A$ is not dense in $[0,1]$ and $T_{b}(A \cap \Q) \subset A$, then $A$ contains at most finitely many $S$-integers.
\end{thm}

\begin{rem} Note  that in Theorem \ref{thm:finite} we only require that $T_{b}(A \cap \Q) \subset A$, which is weaker than that $A$ is $T_b$-invariant. \end{rem} 

Indeed, we have obtained a result on the $\varepsilon$-dense property of orbits of $S$-integers under $T_{b}$, which directly gives Theorem \ref{thm:finite}. 

\begin{thm}\label{thm:dense}
Let $b \geq 2$ be an integer and $S$ be a non-empty finite set of primes not containing any prime divisor of $b$. For any $\varepsilon>0$, there exists an effectively computable positive number $D$, such that for any $\frac{a}{d} \in \Z_{S}\cap [0,1)$ with $(a,d)=1$ and $d>D$, the orbit of $\frac{a}{d}$ under $T_b$, 
\begin{equation}\label{eqorbitTbad}
{\rm Orb}_{T_b}\left(\frac{a}{d}\right):=\left\{T_{b}^{i}\left(\frac{a}{d}\right) \colon i \geq 0 \right\},
\end{equation}
is $\varepsilon$-dense in $[0,1]$.
\end{thm}

\begin{rem}
The number $D$ in Theorem \ref{thm:dense} is given in \eqref{eqdefD}. 
\end{rem}

We also have a quantitative result that strengthens Theorem \ref{thm:finite}. 

\begin{thm} \label{thm:P}
Let $b \geq 2$ be an integer and $A \subset [0,1)$ be a set satisfying $T_b(A\cap \mathbb{Q})\subseteq A$. Suppose $A$ is not dense in $[0,1]$ and let $\varepsilon = \sup\{\dist(x,A) \colon x \in [0,1)\},$ where $\dist(x,A)$ denotes the distance between $x$ and $A$. Then there exists an absolute constant $K>0$ such that for any rational number $\frac{a}{d}$ in $A$ with $\gcd(ab,d) = 1$ and $\varepsilon d \geq 3$, we have 
\begin{align}\label{eqPdbound}
P(d) \geq  \begin{cases} K\sqrt{ \frac{1}{\log{b}} \log{(2\varepsilon d)} \log \log{(2\varepsilon d)}} &\quad\text{ if }\quad P(d) > b,\\
K\sqrt{ \frac{1}{\log{b}} \log{(2\varepsilon d)} } &\quad\text{ if }\quad P(d)<b.
\end{cases} 
\end{align}
\end{thm}

\begin{rem}
{\rm (i)} The absolute constant $K$ in Theorem \ref{thm:P} can be effectively computed.

{\rm (ii)} Note that the assumption that ${\rm gcd}(ab,d)=1$ guarantees that $P(d)\neq b$. 

{\rm (iii)} The assumption that $\varepsilon d \geq 3$ is simply to guarantee that $\log\log(2\varepsilon d)$ is positive. The number $3$ can be slightly decreased if needed.

{\rm (iv)} Theorem \ref{thm:P} can be applied to any generalized Cantor set $C(b,\mathcal{D})$, since such set is $T_{b}$-invariant and $\varepsilon=\frac{m}{2b}$, where $m$ is the largest number of consecutive integers (which may be a single integer) in $\{0,1,\ldots,b-1\} \setminus \mathcal{D}$. Note that our bounds become sharper when $m$ increases, which coincides with the intuition that when $m$ increases there are less rational numbers in $C(b,\mathcal{D})$.  This phenomenon is not reflected in Theorem \ref{thm:Shp}.

\end{rem}

\section{Finiteness of \texorpdfstring{$S$}{Sinteger}-integers}

Let $b \geq 2$ be an integer and $S$ be a non-empty finite set of primes not containing any prime divisor of $b$. In this section, we  prove our $\varepsilon$-dense result Theorem \ref{thm:dense},  and then  deduce the finiteness result Theorem \ref{thm:finite}.

We begin with some notations. For any positive integers $b,L$, we use $\overline{b} \pmod{L}$ to denote the coset in $\Z/L\Z$ containing $b$, and  when $L$ is clear, we simply write it as $\overline{b}$. Let  $(\Z/L\Z)^{\times}$ be the multiplicative group which consists of $\overline{\ell}$ with $\ell$ relatively prime with $L$.  If $b$ and $L$ are coprime, then we denote the order of $\overline{b}$ by $\ord(\overline{b},L)$ and the cyclic subgroup generated by $\overline{b}$ is $\lan \overline{b} \ran$. Equivalently, $\ord(\overline{b},L)$ is the smallest positive integer such that $b^{\ord(\overline{b},L)} \equiv 1 \pmod{L}$. Let $G$ be a finite group, the \textit{exponent} of $G$, denoted by $\exp(G)$, is the smallest positive integer such that $g^{\exp(G)} = 1$ for all $g \in G$.

Recall the following basic result on the multiplicative group  $(\Z/p^n\Z)^{\times}$.

\begin{lem} [{\cite[Chapter 4]{IR90}}] \label{thm:multi_group_strucutre}
For any $n \geq 3$, we have $$(\Z/2^n\Z)^{\times} \cong \lan\overline{-1}\ran \times \lan \overline{5} \ran \cong \Z/2\Z \times \Z/2^{n-2}\Z.$$ 
For any odd prime $p$ and $n \geq 1$, we have $$(\Z/p^n\Z)^{\times} \cong \Z/(p-1)p^{n-1}\Z.$$
\end{lem}

Next lemma gives relation between orders of elements.

\begin{lem}\cite[Proposition 5, Chapter 2, Section 3]{DummitFoote}\label{lem:order}
Let $G$ be a group and $g\in G$ be an element with finite order $s$. Then for each  integer $t\geq1$, the order of $g^t$ is  
\[\frac{s}{{\rm gcd}(s,t)}.\]
\end{lem}

Now we prove a key lemma concerning the order of $\overline{b}$.

\begin{lem} \label{lem:order_d}
Let $b \geq 2$ be an integer and $S$ be a non-empty finite set of primes not containing any prime divisor of $b$. For each $p \in S$, define
$$n_{p} = \begin{cases}
	\max\{3, v_{2}(b-1), v_{2}(b+1)\}, & \text{ if } p=2,\\
	\max\{1, v_{p}(b^{p-1}-1)\}, & \text{ if } p \neq 2,
\end{cases}$$
and
\begin{equation}\label{eqNp}
N_{p} = \max\{n_{p} - v_{p}(\ord(\overline{b},p^{n_{p}})) + v_{p}(\ord(\overline{b},q^{n_{q}})) \colon q \in S \}.
\end{equation}
Then for any integer $d = \prod_{p \in S} p^{e_{p}}$, we have
\begin{equation} \label{eq:order}
\ord(\overline{b},d) =\left(\prod_{p \in S: \ e_{p} > N_{p}} p^{e_{p}-N_{p}}\right) \ord\left(\overline{b},\prod_{p \in S} p^{\min\{e_{p}, N_{p}\}}\right),
\end{equation}
where $\prod_{p \in S:\ e_{p} > N_{p}} p^{e_{p}-N_{p}}$ is defined to be $1$ if $e_{p} \leq N_{p}$ for all $p \in S$.
\end{lem}
\begin{proof}
We start with the case $S=\{2\}$. In this case, we have $N_{2} = n_{2}$. If $e_{2} \leq N_{2}$, then (\ref{eq:order}) is trivial. So it suffices to prove
$$\ord(\overline{b},2^{e}) = 2^{e-n_{2}} \ord(\overline{b}, 2^{n_{2}}) \text{ for any } e > n_{2}.$$
Since $e > n_{2}$, we have $b \not\equiv \pm 1 \pmod{2^e}.$ So Lemma \ref{thm:multi_group_strucutre} implies that $\overline{b} \in \lan \overline{5} \ran \text{ or } \lan \overline{-5} \ran$ in $(\Z/2^{e}\Z)^{\times}$. Let $\overline{g} = \overline{5}$ or $\overline{-5}$ such that $\overline{b} = \overline{g}^{t}$ for some $t \geq 1$. In the group $(\Z/2^{n_{2}+1}\Z)^{\times}$, the element $\overline{g}$ has order $2^{n_{2}-1}$ by Lemma \ref{thm:multi_group_strucutre}, so Lemma \ref{lem:order} and $\overline{b} = \overline{g}^{t}$ imply that $$\ord(\overline{b},2^{n_{2}+1}) \gcd(t, 2^{n_{2}-1}) = 2^{n_{2}-1}.$$ Note that $b \not\equiv 1 \pmod{2^{n_{2}+1}}$, so $\ord(\overline{b},2^{n_{2}+1}) \neq 1$, hence $\gcd(t, 2^{n_{2}-1}) \leq 2^{n_{2}-2}$, and thus $v_{2}(t) \leq n_{2}-2$. Applying the same argument to the groups $(\Z/2^{n_{2}}\Z)^{\times}$ and $(\Z/2^{e}\Z)^{\times}$, we have 
\begin{align*}
\ord(\overline{b},2^{n_{2}}) \gcd(t, 2^{n_{2}-2})= 2^{n_{2}-2},\quad \ord(\overline{b},2^{e}) \gcd(t, 2^{e-2})= 2^{e-2}.
\end{align*}
Now $v_{2}(t) \leq n_{2}-2$ implies that $\gcd(t, 2^{n_{2}-2}) = \gcd(t, 2^{e-2})$, so
$$\ord(\overline{b},2^{e}) = 2^{e-n_{2}} \ord(\overline{b}, 2^{n_{2}}).$$

Next we treat the case $S=\{p\}$ for some odd prime $p$. The proof for this case is a simpler version of the $S=\{2\}$ case. Again, it suffices to prove
$$\ord(\overline{b},p^{e}) = p^{e-n_{p}} \ord(\overline{b}, p^{n_{p}}) \text{ for any } e > n_{p}.$$ 
Let $\overline{g}$ be a generator of $(\Z/p^n\Z)^{\times}$, so $\overline{b} = \overline{g}^{t}$ for some $t \geq 1$. Applying Lemma \ref{thm:multi_group_strucutre} and Lemma \ref{lem:order} to the groups $(\Z/p^{n_{p}}\Z)^{\times}$, $(\Z/p^{e}\Z)^{\times}$ and $(\Z/p^{n_{p}+1}\Z)^{\times}$, we have 
\begin{align*}
\ord(\overline{b},p^{n_{p}}) \gcd(t, (p-1)p^{n_{p}-1}) &= (p-1)p^{n_{p}-1},\\
\ord(\overline{b},p^{e}) \gcd(t, (p-1)p^{e-1}) &= (p-1)p^{e-1},\\
\ord(\overline{b},p^{n_{p}+1}) \gcd(t, (p-1)p^{n_{p}}) &= (p-1)p^{n_{p}}.
\end{align*}
Since $b^{p-1} \not\equiv 1 \pmod{p^{n_{p}+1}}$, we have $p-1 \nmid \ord(\overline{b},p^{n_{p}+1})$. So $p^{n_{p}} \nmid \gcd(t, (p-1)p^{n_{p}})$, and thus $v_{p}(t) \leq n_{p}-1$. Then $\gcd(t, (p-1)p^{n_{p}-1}) = \gcd(t, (p-1)p^{e-1})$ and therefore
$$\ord(\overline{b},p^{e}) = p^{e-n_{p}} \ord(\overline{b}, p^{n_{p}}).$$ 

Finally we consider the general case. We are going to show that 
\begin{align*}
v_{q}(\ord(\overline{b},d)) = v_{q}\left(\prod_{p \in S: \ e_{p} > N_{p}} p^{e_{p}-n_{p}}\right) + v_{q}\left(\ord\left(\overline{b},\prod_{p \in S} p^{\min\{e_{p}, n_{p}\}}\right)\right)
\end{align*}
for all primes $q$, which would imply (\ref{eq:order}).

By the Chinese Reminder Theorem, the map
\begin{align*}
f: \Z/d\Z &\to \prod_{p \in S} \Z/p^{e_{p}}\Z \\
\overline{a} \pmod{d} &\mapsto (\overline{a} \pmod{p^{e_{p}}})_{p \in S}
\end{align*}
is a ring isomorphism. So it induces a group isomorphism 
$$ f: (\Z/d\Z)^{\times} \to \prod_{p \in S} (\Z/p^{e_{p}}\Z)^{\times}.$$
Therefore
\begin{align*}
\ord(\overline{b},d) &= \exp \lan \overline{b} \pmod{d} \ran \\
&= \exp \lan (\overline{b} \pmod{p^{e_{p}}})_{p \in S} \ran \\
&= \lcm\{ \exp \lan \overline{b} \pmod{p^{e_{p}}} \ran \colon p \in S \} \\
&= \lcm\{ \ord(\overline{b},p^{e_{p}}) \colon p \in S \}.
\end{align*}
So for any prime $q$, we have 
\begin{align}
v_{q}(\ord(\overline{b},d)) &= v_{q} (\lcm\{ \ord(\overline{b},p^{e_{p}}) \colon p \in S \}) \nonumber \\ 
&= \max\{v_{q}(\ord(\overline{b},p^{e_{p}})) \colon p \in S \}. \label{eq:LSH}
\end{align}
Similarly, 
\begin{align}
v_{q}\left(\ord\left(\overline{b},\prod_{p \in S} p^{\min\{e_{p}, N_{p}\}}\right)\right) = \max\{ v_{q}(\ord(\overline{b},p^{\min\{e_{p}, N_{p}\}})) \colon p \in S \}. \label{eq:RHS}
\end{align}
For any $p \in S$, since $N_{p} \geq n_{p}$, the cases we have proven imply that \begin{align}
\ord(\overline{b},p^{e_{p}}) = p^{\max\{e_{p}-n_{p},0\}} \ord(\overline{b},p^{\min\{e_{p}, n_{p}\}}) \label{eq:order_pn}
\end{align}
and
\begin{align}
\ord(\overline{b},p^{e_{p}}) = p^{\max\{e_{p}-N_{p},0\}} \ord(\overline{b},p^{\min\{e_{p}, N_{p}\}}). \label{eq:order_pN}
\end{align}

If either $q \notin S$ or $e_{q} \leq N_{q}$, we have $$v_{q}(\ord(\overline{b},p^{e_{p}})) = v_{q}(\ord(\overline{b},p^{\min\{e_{p}, N_{p}\}})).$$
So (\ref{eq:LSH}) and (\ref{eq:RHS}) imply that
$$v_{q}(\ord(\overline{b},d)) = v_{q}\left(\ord\left(\overline{b},\prod_{p \in S} p^{\min\{e_{p}, N_{p}\}}\right)\right).$$
Then (\ref{eq:order}) follows from this since $v_{q}(\prod_{p \in S:\ e_{p} > N_{p}} p^{e_{p}-N_{p}}) = 0$.

If $q \in S$ and $e_{q} > N_{q}$, for any $p \neq q$, we have 
\begin{align*}
v_{q}(\ord(\overline{b},p^{e_{p}})) &= v_{q}(\ord(\overline{b},p^{n_{p}})) \qquad (\text{by } \eqref{eq:order_pn})\\
&\leq N_{q} - n_{p} + v_{q}(\ord(\overline{b},q^{n_{q}}))  \\
&< e_{q} - n_{p} + v_{q}(\ord(\overline{b},q^{e_{q}}))   \\
&= v_{q}(q^{e_{p}-n_{p}} \ord(\overline{b},q^{n_{q}})) \\
&= v_{q}(\ord(\overline{b},q^{e_{q}})), \qquad (\text{by } \eqref{eq:order_pn})
\end{align*}
where  the second inequality follows from the definition of $N_{q}$ and the third inequality holds since $N_{q} < e_{q}$.
So
\begin{align*}
\max\{v_{q}(\ord(\overline{b},p^{e_{p}})) \colon p \in S \} &= v_{q}(\ord(\overline{b},q^{e_{q}})) \\
&= q^{e_{q}-N_{q}} + v_{q}(\ord(\overline{b},q^{N_{q}})).
\end{align*}
Combine this with (\ref{eq:LSH}) and (\ref{eq:RHS}), we deduce (\ref{eq:order}).
\end{proof}

Next result concerns the orbits of $S$-integers under $T_b$, it is the key of this article that leads to the proofs of our main theorems.

\begin{thm} \label{thm:A_1A_2}
Let $b \geq 2$ be an integer, $S$ be a non-empty finite set of primes not containing any prime divisor of $b$, and  $\frac{a}{d}$ be an $S$-integer with $(a,d)=1$ and $d=\prod_{p\in S}p^{e_p}$.  Let $N_p$ be as in \eqref{eqNp}, and set $d_{0} = \prod_{p \in S:\ e_{p} > N_{p}} p^{e_{p}-N_{p}}$, $d_{1} = \prod_{p \in S} p^{\min\{e_{p}, N_{p}\}}$. Define 
\begin{align*}
A_{1} &= \left\{T_{b}^{i}\left(\frac{a}{d}\right): 0 \leq i \leq \ord(\overline{b},d)-1 \right\},\\
A_{2} &= \left\{\frac{1}{d_{0}} T_{b}^{i}\left(\frac{a}{d_{1}}\right) + \frac{j}{d_{0}}: 0 \leq i \leq \ord(\overline{b},d_{1})-1, 0 \leq j \leq d_{0}-1 \right\}.
\end{align*}
Then $A_1 = A_2$.
\end{thm}
\begin{proof}
First we prove $A_{1} \subset A_{2}$. For any $0 \leq i \leq \ord(\overline{b},d)-1$, by the definition of order, there exists $0 \leq i' \leq \ord(\overline{b},d_{1})-1$ such that $b^{i} \equiv b^{i'} \pmod{d_{1}}$. Let $j = \lfloor d_{0} T_{b}^{i}\left(\frac{a}{d}\right) \rfloor$, it is an integer between $0$ and $d_{0}-1$ since $T_{b}^{i}\left(\frac{a}{d}\right) \in [0,1)$. Then
\begin{align*}
d_{0} T_{b}^{i}\left(\frac{a}{d}\right) &= T_{b}^{i}\left(\frac{d_{0} a}{d}\right) + j\\
&= T_{b}^{i}\left(\frac{a}{d_{1}}\right) + j \qquad (\text{since } d=d_{0}d_{1}) \\
&= T_{b}^{i'}\left(\frac{a}{d_{1}}\right) + j.
\end{align*}
So $$T_{b}^{i}\left(\frac{a}{d}\right) = \frac{1}{d_{0}} T_{b}^{i'}\left(\frac{a}{d_{1}}\right) + \frac{j}{d_{0}} \in A_{2}.$$ Since $i$ is arbitrary, we have $A_{1} \subset A_{2}$.

Now we compute the cardinality of $A_{1}$. For any $0 \leq i_{1}, i_{2} \leq \ord(\overline{b},d)-1$, if $T_{b}^{i_{1}}\left(\frac{a}{d}\right) = T_{b}^{i_{2}}\left(\frac{a}{d}\right)$, then $a b^{i_{1}} \equiv a b^{i_{2}} \pmod{d}$. Without loss of generality, we assume $i_{1} \leq i_{2}$. Then since both $a,b$ are coprime with $d$, we have $b^{i_{2}-i_{1}} - 1 \equiv 0 \pmod{d}$ and thus $\ord(\overline{b},d) \mid (i_{2} - i_{1})$, which only happens when $i_{1} = i_{2}$ as $0 \leq i_{2} - i_{1} \leq \ord(\overline{b},d)-1$. Therefore the cardinality of $A_{1}$ equals $\ord(\overline{b},d)$.

Next we compute the cardinality of $A_{2}$. Suppose that 
$$\frac{1}{d_{0}} T_{b}^{i_{1}}\left(\frac{a}{d_{1}}\right) + \frac{j_{1}}{d_{0}} = \frac{1}{d_{0}} T_{b}^{i_{2}}\left(\frac{a}{d_{1}}\right) + \frac{j_{2}}{d_{0}}$$
for some $0 \leq i_{1}, i_{2} \leq \ord(\overline{b},d_{1})-1$ and $0 \leq j_{1}, j_{2} \leq d_{0}-1$. Multiply both sides by $d_{0}$, we have
$$T_{b}^{i_{1}}\left(\frac{a}{d_{1}}\right) + j_{1} = T_{b}^{i_{2}}\left(\frac{a}{d_{1}}\right) + j_{2}.$$
Comparing the integer parts of both sides, we deduce that  $j_{1} = j_{2}$. Then $T_{b}^{i_{1}}\left(\frac{a}{d_{1}}\right) = T_{b}^{i_{2}}\left(\frac{a}{d_{1}}\right).$ By a similar argument as in the previous paragraph, we have $i_{1} = i_{2}$. Therefore the cardinality of $A_{1}$ equals $d_{0} \ord(\overline{b},d_{1})$.

By Lemma \ref{lem:order_d}, we have $\ord(\overline{b},d)$ = $d_{0} \ord(\overline{b},d_{1})$. So $A_{1}$ and $A_{2}$ have the same cardinality and thus $A_{1} = A_{2}$.
\end{proof}

\begin{proof}[Proof of Theorem \ref{thm:dense}]
Let 
\begin{equation}\label{eqdefD}
D= \frac{1}{2 \varepsilon} \prod_{p \in S} p^{N_{p}}.
\end{equation} For any $i \geq 0$, write $i = k \ord(\overline{b},d) + i'$ where $k \in \N$ and $0 \leq i' \leq \ord(\overline{b},d)-1$. Since $b^{\ord(\overline{b},d)} \equiv 1 \pmod{d}$, we have 
$$ T_{b}^{i}\left(\frac{a}{d}\right) = T_{b}^{i'} T_{b}^{k \ord(\overline{b},d)} \left(\frac{a}{d}\right) = T_{b}^{i'} \left(\frac{a}{d}\right).$$
Hence $\Orb_{T_{b}}\left(\frac{a}{d}\right) = A_{1}$, and thus $\Orb_{T_{b}}\left(\frac{a}{d}\right) = A_{2}$ by Theorem \ref{thm:A_1A_2}. Note that $$A_{2} \supset \left\{\frac{1}{d_{0}}\cdot \frac{a}{d_{1}} + \frac{j}{d_{0}} \colon 0 \leq j \leq d_{0}-1\right\},$$ and the distance of any two consecutive elements in the later set is $\frac{1}{d_{0}}.$ By the definition of $d_{0}$, we have 
$$\frac{1}{d_{0}} \leq \prod_{p \in S} p^{N_{p} - e_{p}} = \frac{1}{d} \prod_{p \in S} p^{N_{p}} < \frac{1}{D} \prod_{p \in S} p^{N_{p}} = 2 \varepsilon.$$
Therefore $\Orb_{T_{b}}\left(\frac{a}{d}\right)$ is $\varepsilon$-dense in $[0,1]$.
\end{proof}

We next deduce Theorem \ref{thm:finite} from Theorem \ref{thm:dense}.

\begin{proof}[Proof of Theorem \ref{thm:finite}]
Since $A$ is not dense in $[0,1]$, there exists an interval $I \subset [0,1] \setminus A$ with positive length $\abs{I}$. Let $\varepsilon = \frac{\abs{I}}{2}$, so $A$ is not $\varepsilon$-dense in $[0,1]$. Let $\frac{a}{d} \in A \cap \Q$ with $\gcd(a,d)=1$. Since $T_{b}(A \cap \Q) \subset A$, we have $\Orb_{T_{b}}\left(\frac{a}{d}\right) \subset A$, and so $\Orb_{T_{b}}\left(\frac{a}{d}\right)$ is also not $\varepsilon$-dense in $[0,1]$. Therefore Theorem \ref{thm:dense} implies that $d<D$ for some positive number $D$. Clearly there are only finitely many rational numbers $\frac{a}{d} \in [0,1)$ with $d<D$, hence $A$ contains at most finitely many $S$-integers.
\end{proof}

\section{Largest prime divisor}
In this section, we use Theorem \ref{thm:dense} to prove Theorem \ref{thm:P}. Since we do not intend to compute the exact value of the absolute constant in Theorem \ref{thm:P}, we use the notations $\alpha \gg \beta$ and $\beta \ll \alpha$ to mean that $\abs{\alpha} \geq c \abs{\beta}$ for some absolute constant $c$, which can be effectively computed.

We start with upper bounds on $n_{p}$ and $N_{p}$. 

\begin{lem} \label{lem:bounds_Np}
Keep the notations of Lemma \ref{lem:order_d}. Let $P$ be the largest prime in $S$. Then for any $p \in S$, we have 
\begin{align*}
n_{p} \ll \frac{p\log{b}}{\log{p}} \quad \text{ and } \quad N_{p} \ll \frac{\log{b}}{\log{p}} \left(p + P\right).
\end{align*}
\end{lem}
\begin{proof}
The bound of $n_{p}$ is deduced from its definition and the trivial inequality $$v_{p}(x) \leq \frac{\log{x}}{\log{p}} \text{ for any } x>0.$$

For any $q \in S$, note that $\ord(\overline{b},q^{n_{q}})$ cannot be bigger than the order of $(\Z/q^{n_{q}}\Z)^{\times}$, which equals $(q-1)q^{n_{q}-1}$, so
$$ v_{p}(\ord(\overline{b},q^{n_{q}})) \leq \frac{\log{(q-1)q^{n_{q}-1}}}{\log{p}} \leq \frac{n_{q}\log{q}}{\log{p}} \ll \frac{q \log{b}}{\log{p}}.$$
Then
$$N_{p} \leq n_{p} + \max_{q \in S} v_{p}(\ord(\overline{b},q^{n_{q}})) \ll  \frac{p\log{b}}{\log{p}} + \frac{P \log{b}}{\log{p}}.$$
\end{proof}

Now we are ready to prove Theorem \ref{thm:P}.
\begin{proof}[Proof of Theorem \ref{thm:P}]
Suppose $S$ is the set of all prime divisors of $d$ and the prime decomposition of $d$ is $d= \prod_{p \in S} p^{e_{p}}.$ By Theorem \ref{thm:dense} and the choice of $D$ in its proof (cf. \eqref{eqdefD}), we have $$\log{d}\leq \log D= - \log (2 \varepsilon) + \sum_{p \in S} N_{p} \log{p}.$$
So Lemma \ref{lem:bounds_Np} implies $$\log{(2\varepsilon d)} \ll \sum_{p \in S} ((p+P)\log{b}) \leq (\log{b}) \sum_{p \in S} 2P = 2P\#S\log{b},$$ where $\#S$ denotes the cardinality of $S$. The prime number theorem says $\#S \ll \frac{P}{\log{P}}$, hence we have 
\begin{align}
\log{(2\varepsilon d)} \ll \frac{P^2}{\log{P}} \log{b}. \label{eq:thm:P}
\end{align}

If $P>b$, then $\log \log{(2\varepsilon d)} \leq 2\log P$, and so
\begin{equation}\label{eqPgd}
\log{(2\varepsilon d)} \log \log{(2\varepsilon d)} \ll P^2 \log b.
\end{equation}  
If $P<b$, then it follows from (\ref{eq:thm:P}) and the trivial bound $\frac{1}{\log{P}} \ll 1$ that 
\begin{equation}\label{eqpld}
\log{(2\varepsilon d)} \ll P^2 \log{b}.
\end{equation}  
Now \eqref{eqPgd} and \eqref{eqpld} yield  the desired inequity \eqref{eqPdbound}.
\end{proof}

\section{Further discussion}
In this section, we present some corollaries of our main results, discuss the $T_{b}$-invariant condition, and raise questions for future research.

\subsection{Rational numbers of more general form}

In Theorem \ref{thm:finite}, we concerns rational numbers whose denominators do not share prime divisors with $b$. The following corollary says that if we restrict to rational numbers of the form $\frac{a}{d^n}$, then it is fine to have $\gcd(b,d)>1$, as long as $d$ has a prime divisor not dividing $b$.

\begin{cor}
Let $b\geq 2$ be an integer, $d \geq 2$ be another integer such that there exists at least one prime $p \mid d$ such that $p \nmid b$, and $A$ be a subset of $[0,1)$. If $A$ is not dense in $[0,1]$ and $T_{b}(A \cap \Q) \subset A$, then $A$ contains at most finitely many rational numbers of the form $\frac{a}{d^n}$, $n \in \N$. 
\end{cor}
\begin{proof}
Let $\tilde{d}$ be the largest divisor of $d$ satisfying $\gcd(\tilde{d},b)=1$, that is, $\tilde{d} = d \prod_{p \mid b} p^{-v_{p}(d)}$. If $\frac{a}{d^n} \in A$, we apply $T_{b}$ to it enough times to make the denominator not containing prime divisors of $b$ anymore. In other words, we choose a big enough integer $m$ such that $m v_{p}(b) \geq n v_{p}(d)$ for all $p \mid b$, and then 
$$T_{b}^{m}\left(\frac{a}{d^n}\right) = \frac{\tilde{a}}{\tilde{d}^n},$$ for some integer $\tilde{a}$. Now Theorem \ref{thm:finite} says that  $\tilde{d}^n$ is bounded, and hence $n$ is bounded. Therefore $A$ contains at most finitely many rational numbers of the form $\frac{a}{d^n}$. 
\end{proof}

In Theorem \ref{thm:dense}, the same conclusion holds if we replace $\frac{a}{d}$ in \eqref{eqorbitTbad} by a rational multiple $\frac{aa'}{dd'}$, where $a'\leq d'\in\N$ with $(a',d')=1$. More precisely, we have the following.

\begin{cor}
Let $b \geq 2$ be an integer and $S$ be a non-empty finite set of primes not containing any prime divisor of $b$. Let $\frac{a'}{d'}$ be a rational number in $[0,1)$ with $(a',d')=1$. Then for any $\varepsilon>0$, there exists an effectively computable positive number $D$, such that for any $\frac{a}{d} \in \Z_{S}\cap [0,1)$ with $(aa',dd')=1$ and $d>D$, the orbit 
\begin{equation}
{\rm Orb}_{T_b}\left(\frac{aa'}{dd'}\right)=\left\{T_{b}^{i}\left(\frac{aa'}{dd'}\right) \colon i \geq 0 \right\},
\end{equation}
is $\varepsilon$-dense in $[0,1]$.
\end{cor}

\begin{proof}
Let 
\[S'=S\cup\{p: p\text{ is prime with }p\mid d', p\nmid b\}.\]
Then  there exists $k\geq 0$ such that for any $\frac{a}{d}\in \Z_{S}$, $T_b^k\left(\frac{aa'}{dd'}\right)$ is an $S'$-integer of the form $\frac{\tilde{a}}{d\tilde{d}}$. Then by Theorem \ref{thm:dense} (in which we take $S=S'$), we see that for any $\varepsilon>0$,  there exists a positive number $D$ such that for any $\frac{a}{d} \in \Z_{S}$ with $(aa',dd')=1$ and $d>D$, the orbit ${\rm Orb}_{T_b}\left(\frac{\tilde{a}}{d\tilde{d}}\right)$ is $\varepsilon$-dense in $[0,1]$. It then follows that ${\rm Orb}_{T_b}\left(\frac{aa'}{dd'}\right)$ is also $\varepsilon$-dense in $[0,1]$ as it contains ${\rm Orb}_{T_b}\left(\frac{\tilde{a}}{d\tilde{d}}\right)$.  
\end{proof}

\subsection{An application on \texorpdfstring{$S$-integers}{Sintegers}  }

Theorem \ref{thm:finite} yields the following property of $S$-integers, which says that the orbit of any infinite set of $S$-integers under $T_b$ is dense in $[0,1]$. 

\begin{cor}\label{corSinteger}
Let $b\geq 2$ be an integer, $S$ be a non-empty finite set
of primes not containing any prime divisor of $b$. Let $X\subset \mathbb{Z}_S\cap  [0,1)$ be an infinite subset of $S$-integers. Then the set 
\[{\rm Orb}_{T_b}(X):=\left\{b^kx \ (\bmod 1): x\in X, k\geq 0\right\}\]
is dense in $[0,1]$. 
\end{cor}

\begin{proof}
Observe that 
\[{\rm Orb}_{T_b}(X)=\bigcup_{k=0}^{\infty}T_b^kX\]
and thus ${\rm Orb}_{T_b}(X)$ is $T_b$-invariant.  If ${\rm Orb}_{T_b}(X)$ is not dense in $[0,1]$, then Theorem  \ref{thm:finite} implies that ${\rm Orb}_{T_b}(X)$ contains at most finitely many $S$-integers, which contradicts  that ${\rm Orb}_{T_b}(X)\supseteq X$ and $X$ is an infinite subset of $S$-integers. Hence ${\rm Orb}_{T_b}(X)$ is dense in $[0,1]$. 
\end{proof}

\subsection{The \texorpdfstring{$T_b$}{Tb}-invariant condition}
In Theorem \ref{thm:finite} and Theorem \ref{thm:P}, we require that $T_{b}(A \cap \Q) \subset A$, which is weaker than $A$ being $T_{b}$-invariant. If we know $A$ is actually $T_{b}$-invariant, then all of our results can apply to $S$-integers in $\overline{A}$, the closure of $A$, by noting that $A$ is $T_{b}$-invariant implies $\overline{A}$ is also $T_{b}$-invariant.

On one hand, clearly there exist  $T_{b}$-invariant sets which are  not generalized Cantor sets. On the other hand, Wu \cite{Wu19} showed that every closed $T_{b}$-invariant set can be covered by a generalized Cantor set with similar Hausdorff dimension (see \cite{Falconer90} for definition).

\begin{thm}\cite[Proposition 9.3]{Wu19}\label{thm:Wu}
Let $A \subset [0, 1)$ be a $T_{b}$-invariant set. Then for any $\epsilon > 0,$ there exist $k \in \N$ and a generalized Cantor set $C(b^k,\mathcal{D})$ such that $\overline{A} \subset C(b^k,\mathcal{D})$ and $\dim_{\rm H}(\overline{A}) \geq \dim_{\rm H}(C(b^k,\mathcal{D})) - \epsilon.$
\end{thm}

This result, combining with Theorem \ref{thm:Shp}, leads to another proof of the finiteness of $S$-integers in $T_{b}$-invariant sets, which we now briefly sketch. If $A$ is $T_{b}$-invariant and not dense, then $\overline{A}$ is $T_b$-invariant and $\overline{A}\neq[0,1]$. From the proof of \cite[Proposition 9.3]{Wu19} we see that $\overline{A}$ is contained in a generalized Cantor set $C(b^k,\mathcal{D})$ with $\mathcal{D}\subseteq\{0,1,\ldots,b^k\}$ and $\#\mathcal{D}<b^k$. Then we can apply Theorem \ref{thm:Shp} to deduce that $\Z_{S} \cap C(b^k,\mathcal{D})$ is finite. Therefore the finiteness of $\Z_{S} \cap \overline{A}$ follows since $\Z_{S} \cap \overline{A} \subset \Z_{S} \cap C(b^k,\mathcal{D})$. 

\subsection{Algebraic numbers}
Now we have a decent understanding of rational numbers in $T_{b}$-invariant sets, one may wonder what happens for algebraic numbers. Given an algebraic number $\delta$ with degree at least $2$ and a $T_{b}$-invariant set $A \subset [0,1)$, we ask if the intersection $A \cap \{\frac{a}{\delta^n} \in (0,1) \colon a,n \in \N\}$ is finite. When $A=C$ is the middle-third Cantor set, Mahler \cite{Mahler84} conjectured that all algebraic numbers in $C$ are rational numbers, so $A \cap \{\frac{a}{\delta^n} \in (0,1) \colon a,n \in \N\}$ is actually the empty set. As solving Mahler's conjecture seems out of reach at the moment, determine the finiteness of $A \cap \{\frac{a}{\delta^n} \in (0,1) \colon a,n \in \N\}$ for arbitrary $\delta$ could also be very hard. We wonder if the problem becomes solvable if $\delta$ is restricted to Pisot numbers, which are very close to integers when raising to high powers.

{\noindent \bf  Acknowledgements}. The authors thank Ying Xiong for helpful comments and suggestions.    This research was partially supported by  NSFC 11671151 and 1187120.


\end{document}